\documentclass[12pt]{article}

\usepackage{amssymb}
\usepackage{amsmath} 
\usepackage{amsfonts}

\usepackage{tikz}
\usetikzlibrary{cd}
\usepackage{float}

\usepackage{amsthm}

\DeclareMathAlphabet{\mathpzc}{OT1}{pzc}{m}{it}

\theoremstyle{plain} 
\newtheorem{theorem}{Theorem}[section] 
\newtheorem{lemma}[theorem]{Lemma}

\newtheorem{proposition}[theorem]{Proposition}

\theoremstyle{definition}

\theoremstyle{remark} 
\newtheorem*{remark}{Remark}

\makeatletter
{\newcount\@hour}
{\newcount\@minute}
\def\timenow{\@hour=\time \divide\@hour by 60
\number\@hour:
  \multiply\@hour by 60 \@minute=\time
  \global\advance\@minute by -\@hour
  \ifnum\@minute<10 0\number\@minute\else
  \number\@minute\fi}
\def\ctimenow{\hfil{\tt \jobname.tex, \today~Time: \timenow }\hfil}
      \let\@oddfoot\ctimenow\let\@evenfoot\ctimenow
\makeatother

\begin{document}

\begin{center}{\bf \Large
Comparison of the sets of attractors for systems of contractions and weak contractions
}
\end{center}
\smallskip
\begin{center}
By
\end{center}
\smallskip
\begin{center} Pawe\l{} Klinga and Adam Kwela
\end{center}

\begin{abstract}
For $n,d\in\mathbb{N}$ we consider the families:
\begin{itemize} 
\item $L_n^d$ of attractors for iterated function systems (IFS) consisting of $n$ contractions acting on $[0,1]^d$,
\item $wL_n^d$ of attractors for weak iterated function systems (wIFS) consisting of $n$ weak contractions acting on $[0,1]^d$.
\end{itemize}
We study closures of the above families as subsets of the hyperspace $K([0,1]^d)$ of all compact subsets of $[0,1]^d$ equipped in the Hausdorff metric. In particular, we show that $\overline{L_n^d}=\overline{wL_n^d}$ and $L_{n+1}^d\setminus\overline{L_n^d}\neq\emptyset$, for all $n,d\in\mathbb{N}$. What is more, we construct a compact set belonging to $\overline{L_2^d}$ which is not an attractor for any wIFS. We present a diagram summarizing our considerations.
\let\thefootnote\relax
\let\thefootnote\relax
\end{abstract}

\section{Introduction}

In this paper we are interested in families of attractors for certain iterated function systems as subsets of the the space $K([0,1]^d)$ of compact subsets of $[0,1]^d$ equipped with the Hausdorff metric. This subject has been studied among others in \cite{DS1}, \cite{DS2}, \cite{DS3}, \cite{my} and \cite{SS}.
 
In \cite[Theorem 3.9]{DS1} (see also \cite{DS2}) E. D'Aniello and T.H. Steele proved that the family IFS$^d$ of attractors for iterated function systems acting on a metric space $[0,1]^d,$ $d\in\mathbb{N},$ is a meager $F_\sigma$ subset of the space $K([0,1]^d)$ (similar considerations can be found in \cite[Theorem 5]{SS}). In \cite[Theorem 4.1]{DS3} and \cite[Theorem 4.1]{my}, independently, it is shown that a broader family, the set wIFS$^d$ of attractors for weak iterated function systems (which are systems of weak contractions, i.e., functions $f$ satisfying $d(f(x),f(y))<d(x,y)$ for all $x\neq y$) acting on $[0,1]^d,$ $d\in\mathbb{N},$ is also meager. Actually the proofs show something more: for each $n\in\mathbb{N}$ the set $wL_n^d$ of all attractors for weak iterated function systems consisting of $n$ functions, is nowhere dense in $K([0,1]^d)$. As a consequence we get that $\overline{wL_n^d}$ is a closed nowhere dense subset of $K([0,1]^d)$. 

The main motivation for this paper is to capture in a precise mathematical way the differences and similarities between families IFS$^d$ and wIFS$^d$. We are interested in the question whether $wL_n^d$ and $L_n^d$, which is the family of all attractors for traditional IFSs consisting of $n$ functions, maintain the same topological properties. 

Let ``$A\rightarrow B$" denote ``$A\subsetneq B$". The following diagram summarizes all our results.

\begin{figure}[H]
	\begin{tikzcd}
		\genfrac{}{}{0pt}{}{\displaystyle L_1^d=wL_1^d=\overline{L_1^d}=\overline{wL_1^d}}{\displaystyle =\{\{x\}:x\in[0,1]^d\}} \arrow[d]
		&
		\mbox{\ }
		&	
		\mbox{\ }
		\\
		L_2^d \arrow{d}\arrow{r}
		&
		wL_2^d \arrow{d}\arrow{r}
		&
		\overline{L_2^d}=\overline{wL_2^d} \arrow{d}
		\\
		L_3^d \arrow{d}\arrow{r}
		&
		wL_3^d \arrow{d}\arrow{r}
		&
		\overline{L_3^d}=\overline{wL_3^d} \arrow{d}
		\\
		\vdots \arrow{d}
		&
		\vdots \arrow{d}
		&
		\vdots \arrow{d}
		\\
		L_n^d \arrow{d}\arrow{r}
		&
		wL_n^d \arrow{d}\arrow{r}
		&
		\overline{L_n^d}=\overline{wL_n^d} \arrow{d}
		\\
		\vdots \arrow{d}
		&
		\vdots \arrow{d}
		&
		\vdots \arrow{d}
		\\
		IFS^d \arrow{r}
		&
		wIFS^d \arrow{r}
		&
		\bigcup_{n\in\mathbb{N}}\overline{L_n^d} \arrow{d}
		\\
		\mbox{\ }
		&
		\mbox{\ }
		&
		\overline{IFS^d}=K([0,1]^d)
	\end{tikzcd}
\end{figure}

In the above diagram, all inclusions follow from Remark \ref{r} and Theorem \ref{thmLnW}. All inequalities follow from Theorems \ref{(n+1)-n}, \ref{2-wIFS} and Proposition \ref{p}.

In Section 2 we provide some background needed in our considerations. All main results are proved in Section 3. 

\section{Preliminaries}

By $K(X)$ we denote the collection of all compact subsets of a compact metric space $(X,d)$. Throughout the paper we will assume that $X=[0,1]^d$ for some $d\in\mathbb{N}$.

In \cite{DS2} D'Aniello and Steele proved the following lemma which we will use in the proof of Theorem \ref{thmLnW}.
\begin{lemma}\label{DSlemma42}\cite{DS2}
	Let $(X,d)$ be a compact metric space. If $(E_k)_{k\in\mathbb{N}}$ is a sequence in $K(X)$ such that $\lim_{k\to\infty}d_H(E_k,E) = 0$ and $(S_k)_{ k\in\mathbb{N} }$ is a sequence in $C(X,X)$ converging uniformly to $S$, then $\lim_{k\to\infty} d_H( S_k[E_k], S[E] ) = 0 $.
\end{lemma} 

A function $f\colon X \to X$ is called a contraction on $X$ if one can find a constant $L\in [0,1)$ such that $d(f(x),f(y))\leqslant L\cdot d(x,y)$ for every $x,y\in X$. 

Every finite set $\{ s_1, \dots, s_k \}$ of contractions acting on $X$ will be called an iterated function system (IFS in short). By $\mathcal{S}\colon K(X) \to K(X)$ we denote the Hutchinson operator for $\{ s_1, \dots, s_k \}$, i.e.
$$\mathcal{S}(A) = \bigcup_{i=1}^k s_i[A].$$
If $(X,d)$ is a complete metric space, then $(K(X),d_H)$, where $d_H$ is the Hausdorff metric, is complete as well. Applying the Banach fixed-point theorem, there is a unique attractor for the iterated function system $\{ s_1, \dots, s_k \}$, i.e., a unique compact set $A\in K(X)$ such that $\mathcal{S}(A) = A$. The form of the Hutchinson operator imposes that attractors are self-similar (at least in some sense), so they are often used to describe fractals.

By IFS$^d$ we will denote the set of all attractors for iterated function systems acting on $[0,1]^d$, i.e., a compact set $A\subseteq[0,1]^d$ is in IFS$^d$ if there is an iterated function system $\{ s_1, \dots, s_k \}$ acting on $[0,1]^d$ such that $\mathcal{S}(A) = A$. It is clear that $IFS^d=\bigcup_{n\in\mathbb{N}}L_n^d$, where
$$L_n^d=\left\{A\in K([0,1]^d):\exists_{\{f_1,\ldots,f_n\}\text{- IFS}}A=\bigcup_{i\leq n}f_i[A]\right\}.$$

A function $f:X\to X$ acting on a compact metric space $(X,d)$ is called a weak contraction whenever $d(s_i(x),s_i(y))<d(x,y)$ for all $x,y\in X$, $x\neq y$. Clearly, every contraction is a weak contraction. A finite set $\{s_1,\dots,s_k\}$ of weak contractions acting on $X$ is called a weak iterated function system (wIFS in short). By \cite{E}, if $X$ is compact then, similarly to the case of traditional iterated function systems, every weak iterated function system has a unique attractor. A compact set $A\subseteq X$ satisfying $A=\bigcup_{i=1}^k s_i[A]$ will be called an attractor for $\{s_1,\dots,s_k\}$. Obviously, each attractor of some IFS is an attractor of some wIFS. The reverse inclusion is not true by \cite{NFM}.

Similarly as before, for $n\in\mathbb{N}$ we will denote 
$$wL_n^d=\left\{A\in K([0,1]^d):\exists_{\{f_1,\ldots,f_n\}\text{- wIFS}}A=\bigcup_{i\leq n}f_i[A]\right\}.$$
Then $wIFS^d=\bigcup_{n\in\mathbb{N}}wL_n^d$ is the family of all attractors for weak iterated function systems acting on $[0,1]^d$

In the above definitions of $L_n^d$ and $wL_n^d$ we do not require that functions $f_1,\ldots,f_n$ are pairwise distinct, so $L_n^d\subseteq L_{n+1}^d$ and $wL_n^d\subseteq wL_{n+1}^d$ for all $n\in\mathbb{N}$.

\begin{remark}
\label{r}
Fix any $n,d\in\mathbb{N}$. It is clear from the definitions that:
\begin{itemize}
\item $L^d_{1}=wL^d_{1}=\overline{L^d_{1}}=\overline{wL^d_{1}}$ is the set of all singletons of $[0,1]^d$,
\item $L^d_{n}\subseteq wL^d_{n}\subseteq\overline{wL^d_{n}}$,
\item $L^d_{n}\subseteq L^d_{n+1}\subseteq \bigcup_k L^d_{k}=IFS^d$,
\item $wL^d_{n}\subseteq wL^d_{n+1}\subseteq \bigcup_k wL^d_{k}=wIFS^d$,
\item $IFS^d\subseteq wIFS^d\subseteq\bigcup_k \overline{wL^d_{k}}$,
\item $IFS^d$ is dense in $K([0,1]^d)$ as it contains all finite subsets of $[0,1]^d$.
\end{itemize}
\end{remark}

Although all our results apply to spaces $K([0,1]^d)$ for arbitrary $d\in\mathbb{N}$, in the proof we will mostly work in the one-dimensional case. The following lemma justifies this approach.

\begin{lemma}
\label{l}
For each $n,d\in\mathbb{N}$ the following hold:
\begin{itemize}
\item $A\in L_n^1$ if and only if $A\times\{0\}^{d-1}\in L_n^d$,
\item $A\in wL_n^1$ if and only if $A\times\{0\}^{d-1}\in wL_n^d$,
\item if $A\in \overline{L_n^1}$ then $A\times\{0\}^{d-1}\in \overline{L_n^d}$.
\end{itemize}
\end{lemma}

\begin{proof}
Observe that if $f:[0,1]^d\to[0,1]^d$ is a (weak) contraction then so is $f\upharpoonright[0,1]\times\{0\}^{d-1}$. On the other hand, for each (weak) contraction $f:[0,1]\to[0,1]$ the map $g:[0,1]^d\to[0,1]^d$ given by $g(x_1,\ldots,x_d)=(f(x_1),0,\ldots,0)$ is a (weak) contraction as well. 
\end{proof}

\section{Results}

However there are known examples of attractors for wIFSs, which are not attractors for IFSs (see \cite{NFM}), in our considerations we will need one in $wL_2^d$ for each $d\in\mathbb{N}$. Thus, we start with the following result.

\begin{proposition}
\label{p}
$wL_2^d\setminus IFS^d\neq\emptyset$ for every $d\in\mathbb{N}$.
\end{proposition}

\begin{proof}
By Lemma \ref{l} it suffices to find $X\in wL_2^1\setminus IFS^1$.

Let $f:[0,1]\to[0,1]$ be given by $f(x)=x-x^2$ for all $x\in[0,1]$. Note that $f$ is a weak contraction. Indeed, for all distinct $x,y\in[0,1]$ we have $|f(x)-f(y)|=|x-y|\cdot|1-(x+y)|<|x-y|$.

Define a sequence of intervals by $I_1=[0,\frac{1}{2}]$ and $I_{n+1}=[0,f(\max I_n)]$. Observe that $\lim_n\max I_n=0$. 

Inductively pick points $x_n$ for $n\in\mathbb{N}$ in such a way that for each $n\in\mathbb{N}$ we have:
\begin{itemize}
\item[(a)] $x_n\in I_1$,
\item[(b)] $x_{n+1}<x_n$,
\item[(c)] $x_n-x_{n+1}>x_{n+1}-x_{n+2}$,
\item[(d)] $\{x_i:i\in\mathbb{N}\}\cap(I_n\setminus I_{n+1})$ is finite,
\item[(e)] $k_1=2$ and $k_n>n\cdot(\sum_{i=1}^{n-1}k_i)$, where $k_n=|\{x_i:i\in\mathbb{N}\}\cap(I_n\setminus I_{n+1})|$.
\end{itemize}
Note that item (d) guarantees that $\lim_n x_n=0$. Put $X=\{0\}\cup\{x_n:n\in\mathbb{N}\}$.

We need to show that $X\in wL_2^1\setminus IFS^1$. Firstly, notice that by items (b) and (c) there is a weak contraction $g:[0,1]\to[0,1]$ such that $g(x_n)=x_{n+1}$ for all $n\in\mathbb{N}$. Observe that the unique fixed point of $g$ has to be $0$. Thus, if $h:[0,1]\to[0,1]$ is the function constantly equal to $x_1$ then $X=g[X]\cup h[X]$, i.e., $X\in wL_2^1$. 

On the other hand, assume to the contrary that $f_1,\ldots,f_k:[0,1]\to[0,1]$ are standard contractions such that $X=\bigcup_{i=1}^k f_i[X]$. 

Find $n\in\mathbb{N}$ such that: 
\begin{itemize}
\item[(i)] $n>k$,
\item[(ii)] if $f_i(0)\neq 0$ then $f_i[X]\cap X\cap I_n=\emptyset$,
\item[(iii)] $\max_{i\leq k}Lip(f_i)\cdot|I_n|<|I_{n+1}|$.
\end{itemize}
The second item is possible to achieve as for each $i\leq k$ either $f_i(0)=0$ or $f_i[X]\cap X$ is finite (by $\lim_n x_n=0$). For the third item use the fact that for each $L\in(0,1)$ we have $|f(x)-f(0)|=|x|\cdot|1-x|>L|x-0|$ for any $x\in(0,1-L)$.

We claim that $X\cap I_n\setminus I_{n+1}$ cannot be covered by $f_i[X]$. Indeed, if $f_i(0)\neq 0$ then $f_i[X]\cap X\cap I_n=\emptyset$ by item (ii). What is more, if $f_i(0)=0$ then $f_i[I_n]\subseteq I_{n+1}$ (by item (iii)), so 
$$(X\cap I_n\setminus I_{n+1})\cap\bigcup_{i\leq k}f_i[X]=(X\cap I_n\setminus I_{n+1})\cap\bigcup_{i\leq k}f_i[X\cap I_1\setminus I_n].$$ 
As $|X\cap I_1\setminus I_n|=\sum_{i=1}^{n-1}k_i$, we conclude that 
$$\left|\bigcup_{i\leq k}f_i[X\cap I_1\setminus I_n]\right|\leq k\cdot\sum_{i=1}^{n-1}k_i<n\cdot \sum_{i=1}^{n-1}k_i<k_n=|X\cap(I_n\setminus I_{n+1})|.$$ Thus, $X\notin IFS^1$ and the proof is completed. 
\end{proof}

The rest of our paper is devoted to studies of the families $\overline{L_n^d}$.

\begin{theorem}\label{thmLnW}
$\overline{L_n^d} = \overline{wL_n^d}$ for every $n,d\in\mathbb{N}$.
\end{theorem}

\begin{proof}
As $\overline{L_n^d} \subseteq \overline{wL^d_n}$, it suffices to prove that $wL_n^d \subseteq \overline{L^d_n}$.

	Let us fix $A\in wL^d_n$. There exist weak contractions $f_1, \dots, f_n$ such that
	$$A = \bigcup_{i=1}^n f_i[A].$$
	
	Let us denote $(\mathbb{Q}\cap [0,1])^d \setminus\{(0,\ldots,0),(1,\ldots,1)\}= \{q_i\colon i \in \mathbb{N}\}, q_i\neq q_j$.
	
	For $j\leqslant n, k\in\mathbb{N} $ let us define $g_k^j\colon [0,1]^d\to [0,1]^d$ as a partially linear function connecting the following points: $ (0,f_j(0)), (1,f_j(1)), \dots, (q_i,f_j(q_i)) $ for $i\leqslant k$ (for $d=1$ the function $g_k^j$ is a polygonal chain). You can view $g_k^j$ as a \textit{polygonal approximation} of $f_j$.
	
	We have the following:
	\begin{enumerate}
		\item $g_k^j \xrightarrow{k \to \infty} f_j$,
		\item $g_k^j$ is a contraction with the following Lipschitz constant: 
		$$ \max\{ L_{x,y}^j \colon x,y\in\{ (0,\ldots,0),(1,\ldots,1),q_1, \dots, q_k \} \} ,$$ 
		where $L_{x,y}^j \in (0,1)$ is such that $ d( f_j(x) , f_j(y) ) \leqslant L_{x,y}^j \cdot d(x,y) $.
	\end{enumerate}
	The existence of those $L_{x,y}^j$ follows from each $f_j$ being a weak contraction: since for fixed $x\neq y$ we have $d( f_j(x) , f_j(y) ) < d(x,y)$, then for those specific $x,y$ (rather than globally) we can find a constant $L_{x,y}^j \in (0,1)$ such that $d( f_j(x) , f_j(y) ) \leqslant L_{x,y}^j \cdot d(x,y)$ holds.
	
	For fixed $k$ let $A_k \in K([0,1]^d)$ be such that
	$$A_k = \bigcup_{j=1}^n g_k^j[A_k]. $$
	It follows that $A_k \in L^d_n$. We will show that $A = \lim_k A_k$, i.e. $A \in \overline{L^d_n}$.
	
	Fix $\varepsilon > 0$. We are looking for $k$ such that $d_H(A,A_k)<\varepsilon$.
	
	We will use the Arzelà–Ascoli theorem. The functions $g_k^j$ are uniformly bounded and uniformly equicontinuous, therefore they have a subsequence that is uniformly convergent. However, we know that $(g_k^j)$ is convergent to $f^j$, therefore this particular convergence is uniform.
	
	As $(A_k)_k \subseteq K([0,1]^d)$ and the space is compact, $A_k$ has a convergent subsequence, $( A_{k_m} )_m$. Let us denote its limit by $B$. Applying Lemma \ref{DSlemma42}, we obtain
	$$ \lim_{m\to\infty} d_H \left( g_{k_m}^j[A_{k_m}], f^j[B] \right) = 0.$$
	It is easy to see that $ d_H(A \cup B, C \cup D) \leqslant d_H(A,C)+d_H(B,D) $. Using this, as well as the preceding convergence and the fact that $ \bigcup_j g_{k_m}^j[A_{k_m}] = A_{k_m} $, we have
	$$ \lim_{m\to\infty} d_H \left( \bigcup_{j=1}^n g_{k_m}^j[A_{k_m}], \bigcup_{j=1}^n f^j[B] \right) = 0.  $$
	We also know that $d_H(A_{k_m},B)\to 0$, hence $B = \bigcup_j f^j[B] $ and therefore $B=A$.
	
	It follows that $A$ is the limit of $(A_{k_m}) \subseteq L^d_n$.
\end{proof}

\begin{theorem}
\label{(n+1)-n}
	$L^d_{n+1}\setminus \overline{L^d_n} \neq \emptyset$ for every $n,d\in\mathbb{N}$.
\end{theorem}

\begin{proof}
This proof is based on ideas from \cite[Lemma 4.2]{my}.

Fix $n\in\mathbb{N}$ and let us define $k = n^2+1$. Pick points $a_1,\dots,a_{n+1}, b_1, \dots, b_n$ so that the following conditions are met:
	\begin{itemize}
		\item[1.] $\sum_{i=1}^n b_i + (\sum_{i=1}^{n+1}a_i)(k-1) < 1 $,
		\item[2.] $(k-1) a_{i+1}<a_i \left(\frac{9}{10}\right)^{k-2}$,
		\item[3.] $b_1 = 10 a_1 (k-1) $,
		\item[4.] $b_{i+1}=2b_i $.
	\end{itemize}

Define points $x_{i,j}$ for all $i\leq n+1$ and $j\leq k$ by:
$$x_{1,1}=0,\ x_{i,j+1} = x_{i,j} + a_i\left(\frac{9}{10}\right)^{j-1},\ x_{i+1,1} = x_{i,k} + b_i.$$
Put $y_{i,j}=(x_{i,j},0,\ldots,0)\in[0,1]^d$. Let $F_i = \{ y_{i,j} \colon j \in \{ 1, \dots , k \} \}$ and $F = \bigcup_{i\leq n+1} F_i$. We claim that $F\in L^d_{n+1}\setminus \overline{L^d_n}$.

First we will show that $F\in L^d_{n+1}$. Let $f_1,\ldots,f_{n+1}$ be any contractions such that for each $i\leq n+1$ we have $f_i[F\setminus F_i]=\{y_{i,1}\}$, $f_i(y_{i,k})=y_{i,k}$ and $f_i(y_{i,j})=y_{i,j+1}$, for all $j\leq k-1$. Note that the construction ensures that $d(f_i(x),f_i(y))<d(x,y)$ for all distinct $x,y\in F$. Since $F$ is a finite set, there is $L_i\in(0,1)$ such that $d(f_i(x),f_i(y))\leq L_i d(x,y)$ for all $x,y\in F$. Thus, $f_i\upharpoonright F$ defined in this way can be extended to a contraction $f_i:[0,1]^d\to[0,1]^d$ with the same Lipschitz constant $L_i$. Now it suffices to observe that $F=\bigcup_{i\leq n+1}f_i[F]$.

In order to show that $F\notin \overline{L^d_n}$, put $\delta = \frac{1}{4}\cdot a_{n+1}\left(\frac{9}{10}\right)^{k-2}$ (i.e., $\delta$ is one-fourth of $d(y_{n+1,k-1},y_{n+1,k})$, which is the minimal distance between two points from $F$). We claim that $B(F,\delta)\cap L^d_n=\emptyset$, where $B(F,\delta)$ denotes the ball in $K([0,1]^d)$ of radius $\delta$ centered in $F$. 

Suppose to the contrary that there is $A\in L_n^d$ which is $\delta$-close to $F$. For convenience, put $\widetilde{E}_\delta = \bigcup_{x\in E} B(x,\delta)$ for each $E\subseteq[0,1]^d$ (here $B(x,\delta)$ denotes the ball in $[0,1]^d$ of radius $\delta$ centered at $x$). Then $A\subseteq F_\delta$ and $A\cap B(x,\delta)\neq\emptyset$ for every $x\in F$. Since $A\in L_n^d$, there are contractions $g_1,\ldots,g_n$ such that $A=\bigcup_{i\leq n}g_i[A]$. By the Banach fixed-point theorem, each contraction $g_i$ has a unique fixed point $z_i$. Then $z_i\in A$ as otherwise for $x\in A$ given by $d(z_i,x)=d(z_i,A)>0$ ($d(z_i,A)$ is the distance between the point $z_i$ and the compact set $A$) we would have $d(z_i,g_i(x))=d(g_i(z_i),g_i(x))<d(z_i,x)=d(z_i,A)$ which contradicts $g_i(x)\in A$. Thus, for each $i\leq n$ there is $k_i\leq n+1$ such that $z_i\in \widetilde{\left(F_{k_i}\right)}_\delta$. Thus, there is $m\leq n+1$ with $m\notin\{k_i:i\leq n\}$. We will show that $A\cap \widetilde{\left(F_{m}\right)}_\delta$ cannot be covered by $\bigcup_{i\leq n}g_i[A]$. 

Denote $A_i=A\cap \widetilde{\left(F_{i}\right)}_\delta$ for all $i\leq n+1$. Observe that for all $i$ we have $b_i<b_{i+1}$ (by condition 4.) and $a_i>a_{i+1}$ (by condition 2.). Thus, by condition 3. for each $i\leq n$ and $j\leq n+1$ the set $g_i[A_j]$ cannot intersect two sets of the form $A_{j'}$. In particular, $g_i[A_{k_i}]\subseteq A_{k_i}$. Moreover, using $b_i<b_{i+1}$ once again, observe that if $x\in A_{k_i-1}$ and $y\in A_{k_i}$ are such that $d(x,y)=d(A_{k_i-1},A_{k_i})$ (here $d(A_{k_i-1},A_{k_i})$ is the distance between the compact sets $A_{k_i-1}$ and $A_{k_i}$) then $g_i(x)\in A_{k_i}$. Continuing this reasoning, we obtain that
$$g_i\left[\bigcup_{j\leq k_i}A_j\right]\subseteq A_{k_i}$$ 
for all $i\leq n$. Hence, $A_m\cap g_i[A]=A_m\cap\bigcup_{j>k_i}g_i[A_j]$. Notice also that either $k_i>m$ and $A_m\cap g_i[A]=A_m\cap \bigcup_{j>m}g_i[A_j]$ or $k_i\leq m$. In the latter case we have 
$$g_i[A_{k_i+1}]\subseteq \bigcup_{j\leq k_i}A_{j},$$ 
$$g_i[A_{k_i+2}]\subseteq \bigcup_{j\leq k_i}A_{j}\cup A_{k_i+1},$$ 
$$g_i[A_{k_i+3}]\subseteq \bigcup_{j\leq k_i}A_{j}\cup A_{k_i+1}\cup A_{k_i+2}$$ 
etc. In particular, we can again conclude that $A_m\cap g_i[A]=A_m\cap \bigcup_{j>m}g_i[A_j]$. 
Therefore, 
$$A_m\cap\bigcup_{i\leq n}g_i[A]=A_m\cap\bigcup_{i\leq n}\bigcup_{m<j\leq n+1}g_i[A_j].$$

By the definition of $\delta$, for each $i\leq n$ and $x\in F$ the set $g_i[B(x,\delta)]$ can intersect only one ball of the form $B(y,\delta)$, for $y\in F_m$. Using condition 2. we see that actually for every $j>m$ the whole set $g_i[A_j]$ can intersect only one ball of the form $B(y,\delta)$, for $y\in F_m$. This means that $\bigcup_{i\leq n}\bigcup_{m<j\leq n+1}g_i[A_j]$ can intersect no more than $n^2$ balls of this form. As $k=n^2+1$, at least one ball will not intersect the set $\bigcup_{i\leq n}g_i[A]=A$. Hence, $A$ is not $\delta$-close to $F$. This finishes the proof.
\end{proof}

Recall that $ \overline{L^d_1}$ is the family of all singletons (which is a rather small set). Next result shows that $ \overline{L^d_2}$ is a much richer family.

\begin{theorem}
\label{2-wIFS}
	$ \overline{L^d_2} \setminus wIFS^d \neq \emptyset$ for every $d\in\mathbb{N}$.
\end{theorem}

\begin{proof}
	By Lemma \ref{l}, we can work in the space $[0,1]$. Our goal is to construct a set $X\in \overline{L^1_2} \setminus wIFS^1$. 
	
	We will start the construction with a definition of a sequence $(k_n)$. Put 
	$$ k_1=1 ,\  k_n = n \cdot \left( \sum_{i<n} k_i \right) + 1 + n  .$$
	Now let $(I_n^k)_{n\in\mathbb{N}, k \leqslant k_n}$ be a sequence of closed intervals such that the following conditions hold
	\begin{enumerate}
		\item $\forall_n\:\forall_{i,j\leqslant k_n}\:\: |I_n^i|=|I_n^j|>n\cdot \sum_{m>n, k \leqslant k_m} |I_m^k|. $ 
		\item $\forall_{n,j}\:\: \max(I_{n}^{j+1}) < \min(I_n^{j}) $ and $\forall_n\:\: \max(I_{n+1}^1) < \min(I_n^{k_n}) $.
		\item $\forall_{n,j}\:\: d( I_n^j, I_n^{j+1} ) = |I_n^1|$ and $\forall_n\:\: d( I_{n+1}^1 , I_n^{k_n} ) = | I_{n+1}^1 | $.
		\item Values $|I_1^1|$ and $\max(I_1^1)$ are chosen so that
		$$ \overline{\bigcup_{n\in\mathbb{N}, k \leqslant k_n} I_n^k} = \bigcup_{n\in\mathbb{N}, k \leqslant k_n} I_n^k \cup \{ 0 \} \:\:\text{ and }\:\: \bigcup_{n\in\mathbb{N}, k \leqslant k_n} I_n^k \subseteq [0,1]. $$
	\end{enumerate}
	Let us explain the thinking behind these conditions. The first condition ensures us that the intervals are grouped together in regards to their length, the groups are indexed by the subscript and in each following group the length is much smaller compared to the previous group. The second condition informs us how the intervals are ordered on the line. The third condition assures us that in a particular ``section'' of intervals, the distance between two neighboring intervals is equal to their length. Also, the distance between two sections, i.e. the length of a ``gap'' between two sections, is equal to the length of a gap \textit{inside} the next section. In condition four we take care of the fact that the intervals ``fit into'' a required space and that their boundaries tend to $0$.
	
	Once the preceding construction is finished, we are ready to define the main set of the proof. Put
	$$ X = \bigcup_{n\in\mathbb{N}, k \leqslant k_n} I_n^k \cup \{0\}.$$
	To complete the proof we need to show that $X\notin wIFS^1$ and $X \in \overline{L^1_2}$. We will start with the first part.
	
	Let us assume otherwise, that $X \in wIFS^1$. This implies that there exists a certain number $n$ and such weak contractions $f_1, f_2, \dots, f_n$ that $X$ is the attractor, i.e. $X = f_1[X] \cup \dots \cup f_n[X].$ Since these weak contractions act on intervals, it will be crucial to consider what number of intervals will be present in their images. We claim that the set $f_1[X] \cup \dots \cup f_n[X]$ cannot cover $\bigcup_{k\leq k_n}I_n^k$.
	
	Observe that
	$$ f_1\left[ \bigcup_{m<n \:\land\: k \leqslant k_m} I_m^k \right] $$
	will cover \textit{at most} $k_1 + \dots + k_{n-1} = \sum_{i<n} k_i$ of the intervals $I_n^1, \dots, I_n^{k_n}$. Hence
	$$ f_1\left[ \bigcup_{m<n \:\land\: k \leqslant k_m} I_m^k \right] \cup \dots \cup f_n\left[ \bigcup_{m<n \:\land\: k \leqslant k_m} I_m^k \right] $$
	will cover at most $n\cdot \sum_{i<n} k_i$ of them. Similarly,
	$$ f_1 \left[ I_n^1 \cup \dots \cup I_n^{k_n} \right] $$
	will cover at most \textit{one} interval (due to the fact the gaps are of equal size and the functions are weak contractions) and therefore
	$$ f_1 \left[ I_n^1 \cup \dots \cup I_n^{k_n} \right] \cup \dots \cup f_n \left[ I_n^1 \cup \dots \cup I_n^{k_n} \right] $$
	will cover at most $n$ of them.
	
	Using condition (1) from the construction of $(I_n^k)$, the sets
	$$ \bigcup_{m>n \: \land \: k \leqslant k_m} I_m^k $$
	will not cover a single interval (combined they are shorter), which means that at least one interval was not fully covered. This finishes this part.
	
	Now we will show that $X\in \overline{L^1_2}$.
	
	Fix $\varepsilon > 0$. To finish the proof we will find weak contractions $f_1, f_2$ such that their attractor will be $\varepsilon$-close to $X$.
	
	For convenience we define the operation $\widetilde{\cdot}_\varepsilon$ as follows: $\widetilde{E}_\varepsilon = \bigcup_{x\in E} B(x,\varepsilon)$. Obviously $E\subseteq \widetilde{E}_\varepsilon$. Since $(I_n^k)$ is a sequence of intervals with decreasing lengths, decreasing gaps and approaching 0, for a fixed $\varepsilon$ the set $\widetilde{X}_\varepsilon$ will eventually \textit{swallow} a cofinite collection of intervals tending to 0. More formally, there is $n\in\mathbb{N}$ such that:
	$$ \widetilde{X}_\varepsilon = \bigcup_{m<n,k\leq k_m}\widetilde{(I_m^k)}_\varepsilon\cup\bigcup_{k\leq k_n-1}\widetilde{(I_n^k)}_\varepsilon\cup [0, \alpha), $$
	where $\alpha = \max(I_n^{k_n})+\varepsilon$.
	
	Let us define $f_1$. Start by enforcing the following conditions:
	\begin{itemize}
		\item $ f_1 \left( \max I_1^{1}\right) = \max I_{1}^2 $.
		\item $ f_1 \left( \min I_m^{k_m-1} \right) \in \left( \min I_m^{k_m}, \min I_m^{k_m} + \varepsilon \right) $ for $m\leqslant n$.
		\item $ f_1\left( \min I_m^{k_m} \right) = \min \left( I_{m+1}^1 \right) $ for $m<n$.
		\item $ f_1\left( \max I_m^{k_m} \right) = \max \left( I_{m+1}^1 \right) $ for $m<n$.
	\end{itemize}
	
	We also need to determine what happens in $[ 0,\alpha )$. We define an auxiliary, decreasing sequence $(y_n)$ such that:
	\begin{itemize}
		\item $y_1 = \max{I_n^{k_n}} + \frac{\varepsilon}{2}$.
		\item $ f_1[I_n^{k_n-1}] = \{y_1\} $.
		\item $y_{i+1} =\max\{0,y_i - \frac{\varepsilon}{i+1}\}$.
		\item $ f_1(y_i) = y_{i+1} $.
	\end{itemize}
	These conditions determine the values of $f_1$ on isolated points and in-between them $f_1$ is linear. This concludes the definition of $f_1$.
	
	Note that the sequence $(y_n)$ is constantly equal to zero from some point on, since the series $\sum\frac{\varepsilon}{i+1}$ is divergent. Let $i_0=\min\{i\in\mathbb{N}:y_i=0\}$.
	
	For $f_2$ we require that it is a weak contraction and that the following condition holds:
	$$ f_2\left[ I_1^1 \cup \bigcup_{m\leqslant n, \: k \leqslant k_m } f_1[I_m^k] \cup \{y_1, \dots, y_{i_0}\} \right] = I_1^1, $$
	which means that $f_2$ transfers everything onto $I_1^1$.
	
	We obtain that the set
	$$I_1^1 \cup \bigcup_{m\leqslant n, \: k \leqslant k_m } f_1[I_m^k] \cup \{y_1, \dots, y_{i_0}\}$$
	is an attractor for $\{f_1,f_2\}$ and it is $\varepsilon$-close to $X$.
	
\end{proof}

\newcommand{\nosort}[1]{}

\begin{center}
\flushleft{{\sl Addresses:}} \\
Pawe\l{} Klinga \\
Institute of Mathematics \\
University of Gda\'nsk \\
Wita Stwosza 57 \\
80 -- 952 Gda\'nsk \\
Poland\\
e-mail: pklinga@mat.ug.edu.pl
\end{center}

\begin{center}
\flushleft{{\sl Address:}} \\
Adam Kwela \\
Institute of Mathematics \\
University of Gda\'nsk \\
Wita Stwosza 57 \\
80 -- 952 Gda\'nsk \\
Poland\\
e-mail: akwela@mat.ug.edu.pl
\end{center}

\end{document}